\documentclass[11pt,tbtags,reqno]{article}
\usepackage[margin=0.96in, top=2.2cm,bottom=2.1cm]{geometry}

\usepackage[english]{babel}
\usepackage[T1]{fontenc}
\usepackage{setspace}
\usepackage{mathrsfs}
\usepackage{graphicx}
\usepackage{enumitem}
\usepackage[normalem]{ulem}
\usepackage{float}
\graphicspath{{./images/}}
\usepackage{scrextend}
\deffootnote{0em}{1.6em}{\thefootnotemark.\enskip}
\usepackage{xcolor}

\usepackage{amsmath}
\usepackage[makeroom]{cancel}
\usepackage{amssymb}
\usepackage{amsthm}
\usepackage{mathtools}
\usepackage{graphicx}
\usepackage{euscript,mathrsfs}
\usepackage{bbm}
\usepackage[colorinlistoftodos]{todonotes}
\PassOptionsToPackage{hyphens}{url}
\usepackage[colorlinks=true, allcolors=blue]{hyperref}
\usepackage{titlesec}
\titlelabel{\thetitle.\quad}

\mathtoolsset{showonlyrefs}

\hypersetup{
    colorlinks,
    linkcolor={blue},
    citecolor={blue},
    urlcolor={blue}
}
\usepackage{yfonts}
\usepackage{indentfirst}
\usepackage[parfill]{parskip}
\usepackage[
    citestyle=numeric,
    bibstyle=authoryear,
    dashed=false]{biblatex}
\makeatletter
\input{numeric.bbx}
\makeatother
\renewbibmacro{in:}{}
\DeclareNameAlias{author}{last-first}

\DeclareFieldFormat*{title}{#1}
\DeclareFieldFormat*[book]{title}{\textit{#1}}
\DeclareFieldFormat*{volume}{\textbf{#1}}

\newcounter{counter}
\numberwithin{counter}{section}
\newtheorem{theorem}[counter]{Theorem}
\newtheorem{lemma}[counter]{Lemma}
\newtheorem{definition}[counter]{Definition}

\newtheorem{corollary}[counter]{Corollary}

\numberwithin{equation}{section}
\theoremstyle{definition}
\newtheorem{remark}[counter]{Remark}
\allowdisplaybreaks

\newcommand{\fc}{\mathcal{F}}
\newcommand{\pr}{\mathbb{P}}
\newcommand{\ex}{\mathbb{E}}
\newcommand{\eps}{\varepsilon}

\title{\textsc{A Dynkin Game with Independent Processes and Private Information}}
\author{
\textsc{Georgy Gaitsgori} 
\thanks{ 
\,\textsc{Columbia University, Department of Mathematics, 2990 Broadway, New York, NY 10027, USA} (e-mail: {\it gg2793@columbia.edu})}
\and
\textsc{Richard Groenewald}
\thanks{ 
\,\textsc{Columbia University, Department of Statistics, 1255 Amsterdam Ave, New York, NY 10027, USA} (e-mail: {\it rag2202@columbia.edu})}
}
\date{\today}

\begin{document}
\numberwithin{equation}{section}
\maketitle

\begin{abstract}
We analyze a two-player, nonzero-sum Dynkin game of stopping with incomplete information. We assume that each player observes his own Brownian motion, which is not only independent of the other player's Brownian motion but also not observable by the other player. The player who stops first receives a payoff that depends on the stopping position. Under appropriate growth conditions on the reward function, we show that there are infinitely many Nash equilibria in which both players attain infinite expected payoffs. In contrast, the only equilibrium with finite expected payoffs mandates immediate stopping by at least one of the players. Our results hold in the settings of both pure and mixed strategies.
\end{abstract}



\section{Introduction}

One object of study in the theory of optimal stopping is the so-called \textit{Dynkin game}, introduced in \cite{Dynkin_games}. In such games, two or more players observe some stochastic process, and each of them can terminate the process at some moment to maximize his payoff. The game stops when one player decides to terminate the process, whereby each player receives a payoff (possibly negative) according to the game's rules.

Since their introduction by Dynkin, such games of stopping or ``timing'' have been studied extensively by many authors, and in different settings. For a purely probabilistic approach to such games, we refer the reader to works by Bismut \cite{Bismut}, Krylov \cite{KrylovIzv}, Neveu \cite{Neveu} (section VI.6), and Morimoto \cite{Morimoto}, among others. The analytical theory of stochastic differential games with stopping times was developed in Bensoussan \& Friedman \cite{BenFriFA}, \cite{BenFriTAMS}, and Friedman \cite{Friedman}; see also Krylov \cite{KrylovDokl} and Nagai \cite{Nagai}. For more recent developments, we refer the reader to papers by De Angelis, Ferrari \& Moriarty \cite{AngFerMor}, Ekström \& Peskir \cite{EksPes}, Peskir \cite{PeskirTPA}, and the references therein. 

Dynkin games have also been connected to several other problems. To mention some, Cvitanić \& Karatzas \cite{CviKar} explore connections between Dynkin games and backward stochastic differential equations. Several papers address connections with stochastic control problems, see, e.g., Boetius \cite{Boetius}, De Angelis \& Ferrari \cite{AngFer}, and Karatzas \& Wang \cite{KarWan}. Financial applications of Dynkin games, for instance, option pricing, can be found in Kifer \cite{Kifer} and Hamadène \& Zhang \cite{HamZha}. See also Sîrbu, Pikovsky \& Shreve \cite{SirPikShr} and Sîrbu \& Shreve \cite{SirShr} for applications to the pricing of convertible bonds.

A common feature of the models in all of the above references is that players have complete information about the game. By this, we mean that not only do all players observe the same stochastic process(es), but they also know all other properties of the game, including the distributional properties of the process(es) of interest. Such assumptions are quite restrictive and do not always reflect reality. Therefore, much research activity has recently been directed toward Dynkin games with \textit{incomplete information}. 

One way to relax the \text{complete information} assumption is to consider settings with some uncertainty about the game model. For example, De Angelis, Ekström \& Glover \cite{AngEksGlo} study a two-player zero-sum Dynkin game in which both players observe the same diffusion process, but only one player has information about the drift of this diffusion. Another setting was suggested by Ekström, Glover \& Leniec in \cite{EksGloLen}, who study a two-player zero-sum Dynkin game in which players disagree about the distributional properties of the underlying stochastic process. In \cite{AngEks}, De Angelis \& Ekström study an interesting setup in which players are uncertain about their opponents' existence. For more studies of different setups in this context, we refer the reader to \cite{AngGenVil}, \cite{Gru13}, \cite{Jac22}, among others,  or a very recent result by Pérez, Rodosthenous \& Yamazaki \cite{PerRodYam}, who study an interesting stopping game between two players with asymmetric exercise opportunities.

Another, more ambitious approach toward relaxing the complete information assumption is to assume that each player has his own stream of observations, unobservable by other players. In this case, which we call the \textit{private information case}, each player has to terminate the game based only on his observations and knowledge of the law of the other player's process. As an example of such a Dynkin game, one can consider a competition between two companies who need to be the first to develop and then sell similar goods. In such a situation, each firm faces a dilemma: either to wait longer in order to produce a better product, but potentially lose customers and market share due to the other firm's entry into the market, or to produce a lower-quality product, but be the first entrant in the market. Another example would be a hiring problem, in which two firms test the same candidate for a job and need to decide whether to offer the candidate a contract before the other firm does, which carries the risk of hiring a lower-quality worker. 

Despite their theoretical and practical interest, there are very few results available about such Dynkin games, or, for that matter, about any stopping games with private information. Only one particular class of problems in this direction, the so-called contest models with ranked-based reward structures, has been extensively studied. More precisely, contest models with ranked-based reward structures are stopping games in which the final prize of the game is fixed beforehand, and the winner is determined by the rank of his stopping position relative to other contenders; i.e., the player who stops his process at the highest value wins. Such models with privately observed stochastic processes were first introduced by Seel \& Strack in \cite{SeelStrack13}, and then taken up by the same authors in \cite{SeelStrack16}, where the existence and uniqueness of a Nash equilibrium is proven for an $n$-player stopping game with the stochastic processes given by privately observed Brownian motions with drift. See also papers by Feng \& Hobson \cite{FengHob15}, \cite{FengHob16}, who extend the results of Seel \& Strack to more general diffusions and random initial positions, a paper by Hopenhayn \& Squintani \cite{HopSqu11}, who consider a similar setup in which players privately observe Poisson processes, or a paper by Nutz \& Zhang \cite{NutzZhang23}, who study stopping games with private information in the context of mean-field games.

The main object of study in the aforementioned papers is contests rather than Dynkin games, where an additional significant restriction is the fixed final prize for the winner.
To the best of our knowledge, the only results available about Dynkin games with private information and payoff functions that depend on the value of the stopping position are those in the papers by De Angelis, Merkulov \& Palczewski \cite{AngMerPal} and by Gensbittel \& Grün \cite{GenGru}. In the former, the authors study zero-sum Dynkin games in a very general framework that includes, in particular, asymmetric and private information cases. However, De Angelis et al. consider games in which players are allowed to use randomized stopping times since equilibria might not exist otherwise. In contrast, Gensbittel \& Grün study the more common setup of a two-player zero-sum Dynkin game with pure strategies, in which each player observes his own continuous-time finite state Markov chain and needs to make decisions based on his own stream of observations. 
Gensbittel \& Grün also propose, for future research, a more general setting, with observations modeled by infinite-state Markov processes. 

The present paper studies precisely this case. Namely, we consider a two-player, nonzero-sum Dynkin game with pure strategies, in which each player has access only to his own process, modeled by a Brownian motion. The player who stops first gets a reward based on the stopping position. We note here that our game is non-standard in the sense that the players' reward functions need not satisfy any uniform integrability conditions that are typical for the usual Dynkin games. In particular, under appropriate growth conditions on the payoff function, we show that there are infinitely many Nash equilibria in which both players attain infinite expected payoffs, and there are no equilibria in which one player attains an infinite expected payoff while another player receives a finite expected payoff. Moreover, we show that the only equilibrium with finite expected payoffs for both players mandates immediate stopping by at least one of the players.  
% [INSERT REMARK ABOUT PURE STRATEGIES HERE IMO].
We note that the contrast between these cases is even stronger if one makes the following observation. In the infinite payoff case, there exist equilibria (including ones constructed by us in the proof of Theorem \ref{main_thm}) in which the game will terminate in finite time almost surely. However, in any such equilibrium, the expected duration of the game is infinite. This is due to a celebrated identity of Wald (see, e.g., Problem 3.2.12 in \cite{BMSC}), and further implies that if we restrict the players to stopping times with finite expectations, meaning that the expected duration of the game must be finite, the players will have to terminate the game immediately in equilibrium. Even though it is easy to understand that such pairs of stopping times indeed constitute Nash equilibria, it is more important that we are able to prove the non-existence of other Nash equilibria with finite expected payoffs. 
Moreover, we would like to emphasize that all our results hold if the players are allowed to use mixed strategies, which we discuss in Section \ref{sec_extensions}.

The structure of the paper is as follows. We introduce our model formally in Section \ref{model}. In Section \ref{sec_main_results}, we formulate our main results and outline the proofs, which are carried out in Section \ref{sec_proof_of_main_thm}. In particular, we investigate and classify possible Nash equilibria in the setting of Section \ref{model}, depending on the growth properties of the reward function $f(\cdot)$. To obtain our main results, we reduce the Dynkin game under consideration to an optimal stopping problem of a novel type, with a discount function depending on the distribution of the optimal stopping time. To the best of our knowledge such optimal stopping problems with ``endogenous discounting'', which are of separate and independent interest, have not been studied in the literature. Our methods for studying this problem rely mainly on the asymptotic analysis of the players' reward functions, which we carry out using several tools from stochastic analysis. In Section \ref{sec_extensions}, we discuss several immediate ramifications of our main results, including the classification of all Nash equilibria in the case where $f(\cdot)$ is affine, and establish several extensions, including generalizations of our results to the setting of different reward functions for each player, $n$--player games, and to games which permit mixed strategies. Section \ref{sec_technical_proofs} presents three technical lemmas used in the proof of the main results. We conclude by discussing further potential extensions in Section \ref{sec_future_research}.

\section{Model}\label{model}

We fix a probability space $(\Omega, \fc, \pr)$, and on it two independent standard Brownian motions $W^{(1)}, W^{(2)}$. For $i \in \{1,2\}$, let $\mathbb{F}^{(i)} \coloneqq \{\mathcal{F}_t^{(i)}\}_{t \ge 0}$ denote the natural filtration of $W^{(i)}$. Now assume there are two players, labeled 1 and 2, who play the following game of stopping: player $i$ chooses an $\mathbb{F}^{(i)}$--stopping time $\tau_i: \Omega \to [0, \infty]$ and receives payoff $f(x_i + W^{(i)}_{\tau_i})$ if $\tau_i < \tau_{3-i}$, where $x_i \in \mathbb{R}$ is a starting position and $f: \mathbb{R} \to \mathbb{R}$ a reward function to be specified later. Both players desire to maximize their expected payoffs, given explicitly by 
\begin{equation*}
    \ex \left[f\left(x_i + W^{(i)}_{\tau_i}\right) \left(\mathbf{1}_{\{\tau_i < \tau_{3-i}\}} + \frac{1}{2}\mathbf{1}_{\{\tau_i = \tau_{3-i} < \infty\}}\right)\right], \quad i = 1, 2.
\end{equation*}
Our goal is to find all possible Nash equilibria for such a game. 

To define the notion of Nash equilibrium formally, we introduce the following notation. Let $\mathcal{T}_i$ denote the set of all $\mathbb{F}^{(i)}$--stopping times $\tau_i$ such that $\ex [f(x_i + W^{(i)}_{\tau_i})]$ is well defined. When the players choose the strategies $\tau_1, \tau_2$ (i.e., player $i$ chooses the $\mathbb{F}^{(i)}$--stopping time $\tau_i$, $i = 1, 2$), then the expected payoffs and best possible expected payoffs are, respectively, 
\begin{equation}\label{value_functions_i}
    \begin{split}
    J_i(\tau_i, \tau_{3-i}, x_1, x_{2}) &\coloneqq \ex \left[f\left(x_i + W^{(i)}_{\tau_i}\right) \left( \mathbf{1}_{\{\tau_i < \tau_{3-i}\}} + \frac{1}{2}\mathbf{1}_{\{\tau_i = \tau_{3-i} < \infty\}}\right)\right],\\
    V_i(\tau_{3-i}, x_1, x_2) &\coloneqq \sup\limits_{\tau \in \mathcal{T}_i} J_i(\tau, \tau_{3-i}, x_1, x_2).
    \end{split}
\end{equation}

\begin{definition}
	For fixed starting positions $(x_1, x_2) \in \mathbb{R}^2$, we say that a pair of stopping times $(\tau_1, \tau_2) \in \mathcal{T}_1 \times \mathcal{T}_2$ is a Nash equilibrium for the two-player game started at $(x_1, x_2)$, if 
	\begin{equation}\label{nash_equil_definition}
    	\begin{cases}
    	J_1(\tau_1, \tau_2, x_1, x_2)  = V_1(\tau_{2}, x_1, x_2) \geq J_1(\sigma, \tau_2, x_1, x_2),  \quad \forall \ \sigma \in \mathcal{T}_1,\\
    	J_2(\tau_2, \tau_1, x_1, x_2) = V_2(\tau_{1}, x_1, x_2) \geq J_2(\sigma, \tau_1, x_1, x_2),  \quad \forall \ \sigma \in \mathcal{T}_2.
    	\end{cases}
	\end{equation}
\end{definition}

We introduce the following putative assumptions on the reward function $f(\cdot)$:
\begin{enumerate}[leftmargin=35pt, label = \textbf{(A\arabic*)}]
	\item $f(x) < 0$ for $x < 0$, and $f(x) \ge 0$ for $x \ge 0$ \label{assm_1},
	\item $\liminf_{x \to \infty} \left(f(x) / x^{\gamma}\right) > 0$, for some $\gamma > 0$, \label{assm_2} 
	\item $f(x) \leq kx + b, \, \forall \, x \in \mathbb{R}$, for some $k > 0, b \in \mathbb{R}$, \label{assm_3}
	\item $f(x) = kx, \, \forall \, x \in \mathbb{R}$, for some $k > 0$. \label{assm_4}
\end{enumerate}

We are now ready to state our main results.

\section{Main results}\label{sec_main_results}

\begin{theorem}\label{main_thm}
	Consider the two-player game described in the previous section, started at $(x_1, x_2) \in \mathbb{R}^2$. Then we have the following classification of Nash equilibria:
	\begin{enumerate}[label = \Roman*.]
		\item If $f(\cdot)$ satisfies \ref{assm_1} and \ref{assm_2}, there exist infinitely many Nash equilibria with infinite expected payoffs for both players. In particular, these equilibria can be realized by stopping times that are essentially hitting times of properly scaled square-root boundaries, specified in \eqref{square_root_st_times} below. \label{result_1}
		
		\item If $f(\cdot)$ satisfies \ref{assm_1}, \ref{assm_2} with $\gamma = 1$, and \ref{assm_3}, then there are no Nash equilibria in which one player receives a finite expected payoff, while the other receives an infinite expected payoff. \label{result_2}
		
		\item If $f(\cdot)$ satisfies \ref{assm_1} and \ref{assm_2} with $\gamma \ge 1$, and $\max(x_1, x_2) < 0$, then there exist no Nash equilibria in which both players receive finite expected payoffs. \label{result_3}
		
		\item If $f(\cdot)$ satisfies \ref{assm_4}, then there is the following classification of Nash equilibria in which both players receive a finite expected payoff:
		\begin{enumerate}
			\item There are no such equilibria in which $\pr(\tau_1 > 0) > 0$ and $ \pr(\tau_2 > 0) > 0$; \label{result_4a}
			\item A pair $(\tau_1, \tau_2)$ with $\, \pr(\tau_1 > 0) = 0$ and $\, \pr(\tau_2 > 0) = 0$ is a Nash equilibrium if, and only if, $\min(x_1, x_2) \ge 0$;\label{result_4b}
			\item There exist Nash equilibria in which $\, \pr(\tau_i > 0) > 0$ and $\, \pr(\tau_{3-i} > 0) = 0, \, i = 1, 2$, if, and only if, $$\, \min(x_1, x_2) \le 0 < \max(x_1, x_2).$$ And under this condition, there are infinitely many such Nash equilibria.\label{result_4c}
		\end{enumerate} \label{result_4}
	\end{enumerate}
\end{theorem}
\vspace{0.3cm}

The proof of Theorem \ref{main_thm} is quite long and some parts are tedious. Therefore, we would like to provide key ideas behind the proof and its outline before diving into the details.

The essential idea of the proof is to reduce
the two-player Dynkin game under consideration to a non-standard optimal stopping problem with just one Brownian motion and then analyze that reduced problem. To effect this reduction, we note that the expected payoff in \eqref{value_functions_i} can be written as
\begin{equation}\label{value_func_as_discounting}
    J_i(\tau, \tau_{3-i}, x_1, x_2) = \ex \left[ f\left(x_i + W_\tau^{(i)} \right)  c_i(\tau) \right], \quad i = 1, 2, 
\end{equation}
where 
\begin{equation}\label{def_discount_funct}
    c_i(t) \coloneqq \pr(\tau_{3-i} > t) + \frac{1}{2} \pr(t = \tau_{3-i} < \infty), \quad 0 \le t < \infty,
\end{equation}
and we adopt the convention $c_i(\infty) \equiv 0$. Indeed, the tower property for conditional expectations gives, for any $\tau \in \mathcal{T}_i$, 
\begin{equation}\label{tower_property}
    \begin{split}
        J_i(\tau, \tau_{3-i}, x_1, x_2) 
        &=
        \ex \left[
            f\left(x_i + W^{(i)}_{\tau}\right) \, \ex \left. \left[ 
            \mathbf{1}_{\{\tau < \tau_{3-i}\}} + \frac{1}{2} \mathbf{1}_{\{\tau = \tau_{3-i} < \infty\}} \right| \mathcal{F}_\tau^{(i)} 
            \right] 
            \right] 
            \\ 
            &=
            \ex\left[ 
            f\left(x_i + W_\tau^{(i)} \right) \left. \left( \pr(\tau_{3-i} > t) + \frac{1}{2} \pr(t = \tau_{3-i} < \infty) \right) \right|_{t = \tau}
        \right].
    \end{split}
\end{equation} 
Here the first equality follows from the fact that $W^{(i)}_\tau$ is $\mathcal{F}^{(i)}_{\tau}$--measurable, and the second from the fact that $\tau_{3-i}$ is independent of $\mathcal{F}^{(i)}_{\tau}$. It only remains to take the supremum on both sides. 

As a result of this reduction, to prove statements \ref{result_1}-\ref{result_4} of Theorem \ref{main_thm} it suffices to analyze the problem \eqref{value_func_as_discounting}. We note that the gain function in problem \eqref{value_func_as_discounting} is the product of a \textit{decreasing} discount function $c(\cdot)$ and an \textit{increasing} (at least at infinity under the assumptions \ref{assm_2} or \ref{assm_4}) terminal gain function $f(\cdot)$. 
We balance these conflicting effects by studying the problem \eqref{value_func_as_discounting} using some asymptotic analysis. In particular, we exploit and make precise the following two ideas. 
First, we note that if the discount functions $c_i(\cdot), \, i = 1, 2,$ decrease slowly enough, then players can obtain infinite expected payoffs by choosing hitting times of appropriate boundaries as their stopping times. 
Secondly, we show that under some circumstances, dictated by the game, players are indeed incentivized to choose stopping times whose associated discount functions decrease slowly enough. 

\subsection{Outline and Intuition}

To elaborate more and provide some details and intuition on how the aforementioned ideas work, we give the following outline of the main proof, in which all parts of Theorem \ref{main_thm} are proved sequentially.

 \underline{To prove Part \ref{result_1}}, we explicitly construct a family of symmetric Nash equilibria in which both players have infinite expected payoffs. In these equilibria, players choose stopping times which are effectively hitting times of square-root boundaries with appropriate scaling parameters. It turns out that such stopping times' distributions have heavy tails, and thus the corresponding discount functions are large enough to allow players to get infinite expected rewards. We will prove this part by first formally defining the aforementioned stopping times, then explicitly showing that they allow players to attain infinite expected rewards.

 \underline{The proof of Part \ref{result_2}} is based on the idea that, for a player to achieve an infinite expected payoff, the survival function of his stopping time must necessarily be large enough for the other player to also achieve an arbitrarily large expected payoff. Lemma \ref{unbounded_payoff_lemma}, which we state and prove in Section \ref{sec_technical_proofs}, is the key technical ingredient for this part and quantifies the last statement. The proof of this part proceeds as follows: first, we use a result by Novikov \cite{Novikov}, which gives an estimate on the tail distribution of the quadratic variation of martingales, to get a lower bound on the tail distribution of any stopping time with an infinite reward for a given player. This lower bound then allows us to use Lemma \ref{unbounded_payoff_lemma} to conclude that if one player gets an infinite expected reward, then the other player can also get an arbitrarily large expected reward.

 \underline{The proof of Part \ref{result_3}} is based on ideas similar to those in the proof of part \ref{result_2}, although different techniques are required here. In this part, we observe that if both players start below zero and have finite expected payoffs, their stopping times must again be quite large. Namely, they must be greater than the hitting times of the origin by the processes $x_i + W^{(i)}_t, \, t \ge 0$. However, appealing yet again to Lemma \ref{unbounded_payoff_lemma}, we observe that each of these hitting times is large enough to allow the other player to obtain an arbitrarily large expected payoff under an appropriate strategy. 
 % This leads to an obvious contradiction. 

 \underline{Finally, the proof of Part \ref{result_4}} is again based on similar ideas as the previous parts. We show that if the players do not stop immediately, their optimal stopping times will be large enough to allow their opponent to obtain an arbitrarily large expected payoff. Appropriate lower bounds on the survival functions of the optimal stopping times will be obtained using the assumption \ref{assm_4}.
 In particular, under this assumption, one can show that the players' optimal stopping times must necessarily be bounded from below by the hitting times of some boundaries. One can also show that if the processes $x_i + W^{(i)}_t, \, t \ge 0$, do not hit the boundaries immediately, then with high probability these processes will not hit the boundaries for long enough time as well, again leaving enough time for the other player to achieve arbitrarily large expected rewards. To formalize the last statements, we establish Lemma \ref{boundary_stop_times} and Lemma \ref{devilish_lemma} in Section \ref{sec_technical_proofs}.

\section{Proof of Theorem \ref{main_thm}}\label{sec_proof_of_main_thm}

This section is split into four subsections, corresponding to the proofs of the results \ref{result_1}--\ref{result_4} of Theorem \ref{main_thm}.
	
\subsection{Proof of \texorpdfstring{\hyperref[result_1]{I}}{}}\label{subsec_proof_1}
    In this part, we explicitly construct a family of symmetric Nash equilibria in which both players achieve infinite expected payoffs.  Due to the representation of the expected payoffs \eqref{value_func_as_discounting}, it is sufficient to find stopping times $\tau_i \in \mathcal{T}_i, \, i = 1, 2,$ such that 
	\begin{equation}\label{infinite_payoffs_condition}
    	J_i(\tau_i, \tau_{3-i}, x_1, x_2) = \ex \left[  f\left(x_i + W^{(i)}_{\tau_i} \right) c_i(\tau_i) \right] = \infty, \quad i = 1, 2.
	\end{equation} 
	In this case, neither player has an incentive to deviate from his strategy, and thus such a pair $(\tau_1, \tau_2)$ is an equilibrium.  
	
	As is common in optimal stopping problems, we focus our attention on hitting times of boundaries. We start by defining $\sigma^{(i)}_\xi \coloneqq \nobreak \inf \{t \ge\nobreak 0 : W_t^{(i)} = \xi\}, \, \xi \in \mathbb{R}$, and, for a fixed $a > 0$, letting 
	\begin{equation}\label{square_root_st_times}
    	\begin{split}
    	\tau_a^{(i)} 
    	&\coloneqq 
    	\inf \left\{s \ge \sigma_{-x_i}^{(i)} : x_i + W_s^{(i)} \ge a \sqrt{s-\sigma_{-x_i}^{(i)} + 1} \right \} \\ 
    	&= 
    	\sigma_{-x_i}^{(i)} + \inf\left\{t \ge 0: W_{\sigma_{-x_i}^{(i)} + t}^{(i)} - W_{\sigma_{-x_i}^{(i)}}^{(i)} \ge a \sqrt{t+1}\right \} \\ 
    	&= 
    	\sigma_{-x_i}^{(i)} + \widetilde{\tau}_a^{(i)},
    	\end{split}
	\end{equation}
	where 
	\begin{equation}\label{stopping_time_tilde}
    	\widetilde{\tau}_a^{(i)} \coloneqq \inf\left\{t \ge 0: \widetilde W_t^{(i)} \ge a \sqrt{t+1}\right \}
	\end{equation}
	and $\,\widetilde W_t^{(i)} \coloneqq W_{\sigma_{-x_i}^{(i)} + t}^{(i)} - W_{\sigma_{-x_i}^{(i)}}^{(i)}, \,t \ge 0$ is a standard Brownian motion, independent of $\mathcal{F}_{\sigma_{-x_i}^{(i)}}^{(i)}$.
    One may verify that $\tau_a^{(i)}$ is a stopping time, by noting that, for every $t \in [0, \infty)$, we have
	\begin{equation*}
	   \left\{ \tau_a^{(i)} \leq t \right \} 
       = 
       \bigcup_{\substack{r \in \mathbb{Q} : \\ r \leq t}}\left\{\sigma^{(i)}_{-x_i} < r \right\} \cap \bigcap_{\substack{q \in \mathbb{Q}: \\ q > 0}} \bigcup_{\substack{s \in \mathbb{Q}: \\ s \in \left[r, t \right]}} \left\{x_i + W_s^{(i)} \ge a \sqrt{s- \sigma_{-x_i}^{(i)} + 1} - q \right\} \in \mathcal{F}_t^{(i)}.
	\end{equation*}
	Note that, despite its seemingly complicated form, the stopping time $\tau_a^{(i)}$ in \eqref{square_root_st_times} is just the first time the process $x_i + W^{(i)}_t$ hits the square-root boundary $a\sqrt{t+1}$ after hitting the origin. The significance of such stopping times will become clear shortly.
	
	We now claim that, for $a > 0$ large enough, the pairs $(\tau^{(1)}_{a}, \tau^{(2)}_{a})$ satisfy \eqref{infinite_payoffs_condition}. This will follow from Theorem 1 in Breiman \cite{Breiman}, which sheds light on the asymptotic behavior of hitting times of scaled square-root boundaries by a reflected Brownian motion (see also Theorem \ref{breiman_result} in Section \ref{sec_technical_proofs}, where we have stated a precise formulation of this result for the sake of completeness). Indeed, from this result, it follows that there exist a constant $c > 0$ and a function $\beta : \mathbb{R} \to (0, \infty)$ such that for $t$ large enough we have
	\begin{equation}\label{breiman_conditions}
    	\pr(\widetilde{\tau}_a^{(i)} > t - 1) \geq c \, t^{-\beta(a)}  \text{ \ with \ } \lim_{a \rightarrow \infty} \beta(a) = 0.
	\end{equation}  
	Thus, we can choose $a$ large enough such that $\beta(a) < \frac{\gamma}{4}$, where $\gamma$ is the constant from Assumption \ref{assm_2}. Note that $\tau_a^{(i)} < \infty$ holds, almost surely, by the Law of the Iterated Logarithm.  Thus, we have 
	\begin{equation} \label{infinite_objective}
    	\begin{split}
    	J_i(\tau_a^{(i)}, \tau_a^{(3-i)}, x_1, x_2) &= 
    	\ex \left[ f \left(a \sqrt{\widetilde{\tau}_a^{(i)} + 1} \right) c_i\left(\sigma_{-x_i}^{(i)} + \widetilde{\tau}_a^{(i)}\right)\right] \\ 
    	&= 
    	\ex \left[\ex \left. \left[f \left(a \sqrt{\widetilde{\tau}_a^{(i)} + 1} \right) c_i\left(u + \widetilde{\tau}_a^{(i)}\right) \right] \right|_{u = \sigma_{-x_i}^{(i)}} \right],
    	\end{split}
	\end{equation} 
	where the second equality follows from the tower property and the fact that $\sigma_{-x_i}^{(i)}$ is independent of $\widetilde{\tau}_a^{(i)}$. 
	
	We now claim that, for every fixed $u \ge 0$, we have
	\begin{equation*}
	   \ex \left[f \left(a \sqrt{\widetilde{\tau}_a^{(i)} + 1} \right) c_i\left(u + \widetilde{\tau}_a^{(i)}\right) \right] = \infty.
	\end{equation*} 
	To see this, note from \eqref{def_discount_funct} that $c_i(t) \ge \pr \left(\widetilde{\tau}_a^{(3-i)} > t \right)$.  Hence, by \eqref{breiman_conditions}, \ref{assm_1} and \ref{assm_2}, there exist constants $L > 0, M > 0$, which depend only on $a, u, c$ and the function $f(\cdot)$, such that 
	\begin{equation}\label{payoff_under_root_coundary}
    	\begin{split}
    	\ex \left[f \left(a \sqrt{\widetilde{\tau}_a^{(i)} + 1}\,  \right) c_i\left(u + \widetilde{\tau}_a^{(i)}\right) \right] &\ge
    	L \, \ex \left[ \left(\widetilde{\tau}_a^{(i)}\right)^{\frac{\gamma}{2}}\left(\widetilde{\tau}_a^{(i)}\right)^{-\beta(a)} \mathbf{1}_{\left\{\widetilde \tau_a^{(i)} > M \right\}}\right] \\ 
    	% &=
    	% L \int_0^\infty \pr \left( \left(\widetilde{\tau}_a^{(i)}\right)^{q}
    	% \mathbf{1}_{\left\{\widetilde \tau_a^{(i)} > M \right\}} > x \right) dx \\
    	&=
    	L \int_0^\infty \pr \left( \widetilde{\tau}_a^{(i)} > M \vee  x^{1/q} \right) dx,
    	\end{split}
	\end{equation}  
	where 
	$$
	q \coloneqq \frac{\gamma}{2}-\beta(a) > 0.
	$$
	Finally, we apply \eqref{breiman_conditions} again to note that there exist some constants $L' > 0$, $R > M^{q}$, again depending only on $a, u, c, f(\cdot)$, for which the quantity in \eqref{payoff_under_root_coundary} is bounded from below by 
	\begin{equation}\label{infinite_lower_bound}
	L' \int_{R}^\infty \pr \left(\widetilde 
	\tau_a^{(i)} > x^{1/q} \right) dx 
	\ge
	L' \int_R^\infty \left(\frac{1}{x} \right)^{r(a)} dx = \infty, 
	\end{equation}  
	with
	$$
	r(a) \coloneqq \frac{2\beta(a)}{\gamma - 2\beta(a)} < 1
	$$ 
	by the choice of $a > 0$.  Combining this with \eqref{infinite_objective} concludes the proof of part \ref{result_1} \hfill \qed {\parfillskip0pt\par}
    
\subsection{Proof of \texorpdfstring{\hyperref[result_2]{II}}{}} \label{subsec_proof_2}
In this part, we show that there are no Nash equilibria in which one player receives a finite expected payoff, while the other receives an infinite expected payoff.
    We argue by contradiction. Suppose, without loss of generality, that $(\tau_1, \tau_2)$ is a Nash equilibrium where player $1$ obtains an infinite expected payoff while player $2$ obtains a finite expected payoff.  We first show that the stopping time $\tau_1$ of player $1$ is then ``large'', in the sense that the function $c_2(\cdot)$, defined by \eqref{def_discount_funct} in terms of the distribution of $\tau_1$, satisfies
	\begin{equation} \label{liminf_condition}
	\liminf_{t \rightarrow \infty} \left(c_2(t) \sqrt{t}\right) = \infty.
	\end{equation}  
	
	To see this, we note first that \eqref{liminf_condition} is trivially true if $\pr(\tau_1 = \infty) > 0$, since 
	\begin{equation*}
	c_2(t) \ge \pr(\tau_1 > t) \ge \lim_{t \rightarrow \infty} \pr(\tau_1 > t) = \pr(\tau_1 = \infty).
	\end{equation*}
	Thus, let us assume that $\tau_1 < \infty$ holds almost surely. Then \eqref{liminf_condition} follows from a result of Novikov \cite{Novikov} (see also \cite{ElLiYor}). In particular, Theorem 1 of \cite{Novikov} states that for any continuous local martingale $M$ with $\ex[M_\infty^+] < \infty$ or $\ex [M_\infty^-] < \infty$, whose quadratic variation $\langle M \rangle$ satisfies $\langle M \rangle_\infty < \infty$ almost surely, we have the following lower bound:
	\begin{equation}\label{Novikov_liminf}
	   \liminf_{t \rightarrow \infty} \left(\pr\big(\langle M \rangle_\infty > t\big) \,\sqrt{t}\right) \ge \sqrt{\frac{2}{\pi}} \, \Big | \ex [M_\infty] \Big |.
	\end{equation} 
	In our case, we consider the martingale given by the stopped Brownian motion $M_t \coloneqq x_1 + W^{(1)}_{\tau_1 \wedge t}, \, t \ge 0,$ with quadratic variation $\langle M \rangle_t = t \wedge \tau_1$, $t \ge 0$, and note that, by Assumptions \ref{assm_1} and \ref{assm_3}, along with the assumption that player 1 attains an infinite expected payoff, we have 
    \begin{equation} \label{upperboundminus}
    	\ex \left[\left(x_1 +W_{\tau_1}^{(1)} \right)^{-}\right] \leq \frac{1}{k} \left(\ex \left[\left(f\left(x_1 +W_{\tau_1}^{(1)}\right)\right)^{-}\right] + |b| \right) < \infty,
	\end{equation} 
	which makes Novikov's results applicable. Observing that
	\begin{equation} \label{lowerboundplus}
    	\ex \left[\left(x_1 +W_{\tau_1}^{(1)} \right)^{+}\right] \ge \frac{1}{k} \left(\ex \left[\left(f\left(x_1 +W_{\tau_1}^{(1)}\right)\right)^{+}\right] - |b| \right) = \infty,
	\end{equation} 
	we obtain the lower bound 
	\begin{equation*}
        \begin{split}
        	\liminf_{t \rightarrow \infty} \left(c_2(t) \sqrt{t}\right) &\ge \liminf_{t \rightarrow \infty}\left(\pr(\tau_1 > t) \sqrt{t}\right) \\*
        	&\geq \sqrt{\frac{2}{\pi}} \Big| \ex \left[x_1  + W_{\tau_1}^{(1)} \right]\Big|  \\
        	& = \infty.
        \end{split}
	\end{equation*} 
	Here the inequality in the second line is implied by \eqref{Novikov_liminf}.  
	
	As a result, we obtain \eqref{liminf_condition}, which in turn implies \eqref{condition_on_disc_func}. Therefore, we can apply Lemma \ref{unbounded_payoff_lemma}, whose conclusion contradicts the assumption that $(\tau_1, \tau_2)$ is a Nash equilibrium with finite expected payoff for the second player. This completes the proof of part \ref{result_2} \hfill \qed {\parfillskip0pt\par}
    
\subsection{Proof of \texorpdfstring{\hyperref[result_3]{III}}{}} \label{subsec_proof_3}
    In this part, we show that if $f(\cdot)$ satisfies \ref{assm_1} and \ref{assm_2} with $\gamma \ge 1$, and $\max(x_1, x_2) < 0$, then there exist no Nash equilibria in which both players receive finite expected payoffs. 
    The proof is quite technical, so we divide it into several steps. First, we prove the intuitively clear statement, that no player would stop in a negative reward region if there is a chance that the other player has not stopped yet. Formally, this means that for each $i = 1, 2$, the event $\{x_i + W^{(i)}_{\tau_i} < 0\} \cap \{ c_i(\tau_i) > 0\}$ has zero probability. Secondly, we show that this implies $c_i(t) > 0$ for all $t \geq 0$. The third step is to show that the last observation, in turn, implies that the stopping times $\tau_i$ are greater than the hitting times of the origin by the processes $x_i + W^{(i)}_\cdot$. Finally, we use Lemma \ref{unbounded_payoff_lemma} again to show that such ``large'' stopping times allow the other player to obtain an arbitrarily large expected payoff.
	
	\underline{Step 1:} Suppose that $(\tau_1, \tau_2)$ is a Nash equilibrium in which both players receive a finite expected payoff. To show that, for each $i = 1, 2$, we have
	\begin{equation}\label{prob_zero_intersection}
    	\pr\left( \{x_i + W^{(i)}_{\tau_i} < 0\} \cap \{ c_i(\tau_i) > 0\}\right) = 0
	\end{equation}
	in the notation of \eqref{def_discount_funct}, we suppose for the sake of contradiction that this is not true for the first player. However, we then claim that there exists a strategy that is strictly better than $\tau_1$ for the first player. Namely, we consider an alternative stopping time $\widetilde{\tau}_1 \coloneqq \inf \{t \ge \tau_1: x_1 + W_t^{(1)} \ge 0\}$ and note from the representation of the expected payoff \eqref{value_func_as_discounting} that the payoff of the first player under $\widetilde{\tau}_1$ can be written as 
	\begin{equation}\label{comparing_payoffs}
    	\begin{split}
    	J_1(\widetilde{\tau}_1, \tau_{2}, x_1, x_2) 
    	&= 
    	\ex \left[f\left(x_1 + W_{\widetilde{\tau}_1}^{(1)}\right)  c_1(\widetilde{\tau}_1) \right] \\
    	&\ge
    	\ex \left[f\left(x_1 + W_{\widetilde{\tau}_1}^{(1)}\right) c_1(\widetilde{\tau_1}) \, \mathbf{1}_{\{x_1+W_{\tau_1}^{(1)} \ge 0\}} \right] \\
    	&=
    	\ex \left[f\left(x_1 + W_{\tau_1}^{(1)}\right) c_1(\tau_1) \, \mathbf{1}_{\{x_1+ W_{\tau_1}^{(1)} \ge 0\}} \right]\\ 
    	&>
    	\ex \left[f\left(x_1 + W_{\tau_1}^{(1)}\right) c_1(\tau_1) \right]\\
    	&=
    	J_1(\tau_1, \tau_{2}, x_1, x_2).
    	\end{split}
	\end{equation}
	Here the inequality on the second line follows from the fact that $f(x_1 + W_{\widetilde{\tau}_1}^{(1)})$ is non-negative by the choice of $\widetilde{\tau}_1$. The equality on the third line follows since $\tau_1 = \widetilde{\tau}_1$ holds on the set $\{x_1+ W_{\tau_1}^{(1)} \geq 0\}$. The last strict inequality follows from the fact that, on the set $\{x_1 + W^{(1)}_{\tau_1} < 0\} \cap \{ c_1(\tau_1) > 0\}$ of positive probability, the payoff is negative. Therefore, the strategy $\widetilde{\tau}_1$ is strictly better than $\tau_1$, which leads to the desired contradiction.
	
	\underline{Step 2:} Next, we show that for each $i = 1, 2$, we have $c_i(t) > 0$ for all $t \ge 0$ in \eqref{def_discount_funct}. Note that if, for some $t$, we have $c_i(t) = 0$, then $\pr(\tau_{3-i} > t) = 0$ or equivalently $\pr(\tau_{3-i} \le t) = 1$. Thus, to prove that $c_i(t) > 0$ for all $t \ge 0$, we denote by
	\begin{equation*}
    	T_i \coloneqq ||\tau_i||_\infty =  \inf\{t \ge 0 : \pr(\tau_i \leq t) = 1 \}, \quad i = 1, 2,
	\end{equation*}
	the essential least upper bounds of the corresponding stopping times, and show that $T_i = \infty$. 
	
	We argue this by contradiction: assuming that $T_1 < \infty$ (clearly, the case $T_2 < \infty$ can be treated the same way), we obtain
	\begin{equation}\label{negativity_at_stopping}
    	\pr\left(x_{1} + W_{\tau_{1}}^{(1)} < 0\right) \ge 
    	\pr\left(\sup_{0 \leq s \leq T_1} W_{s}^{(1)} < -x_{1} \right) > 0,
	\end{equation}
	since by our assumption $x_{1} < 0$. The arguments in step 1 then imply $c_1(T_1) = 0$, because otherwise, since $\tau_1 \le T_1$ almost surely, we would have $\pr(c_1(\tau_1) > 0) = 1$, which, together with \eqref{negativity_at_stopping}, contradicts \eqref{prob_zero_intersection}. However, the equality $c_1(T_1) = 0$ in turn implies that $\tau_2$ is supported on $[0, T_1)$. By symmetry we obtain $T_1 = T_2$ and that $\tau_1$ must also be supported on $[0, T_1)$. However, this implies $\pr(c_1(\tau_1) > 0) = 1$ since, for almost every $\omega \in \Omega$, we have $\tau_1(\omega) < T_1 = T_2$ and thus $c_1(\tau_1(\omega)) \ge \pr(\tau_2 > t)\big|_{t = \tau_1(\omega)} > 0$. Therefore, the set $\{x_1 + W^{(1)}_{\tau_1} < 0\} \cap \{ c_1(\tau_1) > 0\}$ has positive probability. By \eqref{prob_zero_intersection} of step 1, this leads to the desired contradiction.

	\underline{Step 3:} Combining steps 1 and 2 we obtain 
	$$
	\pr\left(x_i + W_{\tau_{i}}^{(i)} < 0,\, \tau_i < \infty\right) = 0
	$$
	for each $i = 1, 2$. However, this clearly implies that each $\tau_i$ is almost surely greater than or equal to the corresponding hitting time of the origin by the process $x_i + W_{t}^{(i)}, \, t \ge 0$. Using the definitions of the discount functions \eqref{def_discount_funct} and the stopping times 
    $\sigma^{(i)}_\xi = \nobreak \inf \{t \ge\nobreak 0 : W_t^{(i)} = \xi\}, \, \xi \in \mathbb{R}$, we have $\tau_i \ge \sigma^{(i)}_{-x_i}$ almost surely, in particular
	\begin{equation}\label{surv_func_lower_bound}
    	c_{3-i}(t) \geq \pr\left(\tau_{i} > t\right) \ge \pr\left(\sigma^{(i)}_{-x_i} > t\right).
	\end{equation}
	And a lower bound on the tail distribution of $\sigma^{(i)}_{-x_i}$ can be obtained by the reflection principle and Brownian scaling (see, e.g., Chapter 2.6 in \cite{BMSC}):
	\begin{equation}\label{hitting_time_tail_bound}
    	\begin{split}
    	\pr\left(\sigma^{(i)}_{-x_i} > t\right) 
    	&= 
    	2\, \pr\left(W^{(i)}_t \le -x_i\right) - 1\\
    	&=
    	2\, \pr\left(0 \leq W_t^{(i)} < -x_i\right)\\
    	&=
    	2\, \pr \left(0 \leq W_1^{(i)} < \frac{-x_i}{\sqrt{t}} \right)\\
    	&\ge
    	\sqrt{\frac{2}{\pi}} \exp \left(-\frac{x_i^2}{2t} \right) \frac{|x_i|}{\sqrt{t}}\\
    	&\ge
    	\frac{|x_i|}{\sqrt{\pi t}}
    	\end{split}
	\end{equation}
	for $t$ large enough, where the second--to--last inequality follows by the fact that $W^{(i)}_1$ is a standard Gaussian random variable.
	It only remains to show that such a lower bound for each player actually allows the other player to obtain an arbitrarily large expected payoff.
	
	\underline{Step 4:} To make the last statement precise, we apply Lemma \ref{unbounded_payoff_lemma} once again. Note that the bounds \eqref{surv_func_lower_bound} and \eqref{hitting_time_tail_bound} immediately lead to \eqref{condition_on_disc_func}, so the conclusion of the Lemma \ref{unbounded_payoff_lemma} contradicts the assumption that $(\tau_1, \tau_2)$ is a Nash equilibrium with finite expected payoffs for both players. \hfill \qed {\parfillskip0pt\par}
	
\subsection{Proof of \texorpdfstring{\hyperref[result_4]{IV}}{}} \label{subsec_proof_4}
Since Part \ref{result_4} consists of three statements, each with its own distinct proof; we present them sequentially.
	
	%The main idea in the proof of part \ref{result_4a} is essentially the same as in the proof of parts \ref{result_2} and \ref{result_3}. If the players do not stop immediately, their optimal stopping time will be large enough to allow the other player to obtain an arbitrarily large expected payoff. Appropriate lower bounds on the survival functions of the optimal stopping times in part \ref{result_3} were obtained by using the negativity of the initial positions. Another condition that can lead to the same conclusions is the assumption \ref{assm_4} on the reward function $f(\cdot)$.
	
%	We first establish the following important result, which states that optimal stopping times for both players in any Nash equilibrium are bounded from below by first hitting times of some boundaries when the function $f(\cdot)$ is linear. 
	
%	\begin{lemma}\label{boundary_stop_times_}
%		Consider the optimal stopping problem of \eqref{value_func_in_unbounded_lem} with reward function $f(\cdot)$ and discount function $c(\cdot)$. Assume that $f(\cdot)$ satisfies assumption \ref{assm_4} and that there exists an optimal stopping time for the problem under consideration. Then any optimal stopping time $\tau^*$ is bounded from below by the hitting time
%		\begin{equation}\label{boundary_condition_}
%		\tau = \inf \left\{t \ge 0 : x + W_t \ge b(t) \right\}
%		\end{equation}
%		of some boundary $b: [0, \infty) \to [-\infty, \infty]$.
%	\end{lemma}
	
	%Now 

	\noindent \textbf{Proof of \ref{result_4}a:} In this part we show that if $f(\cdot)$ satisfies \ref{assm_4}, then there are no equilibria with finite expected payoffs, in which $\pr(\tau_1 > 0) > 0$ and $ \pr(\tau_2 > 0) > 0$.
    Let us assume that $(\tau_1, \tau_2)$ is a Nash equilibrium in which both players receive a finite expected payoff, and introduce again
	\begin{equation}\label{stopping_times_sups}
    	T_i \coloneqq ||\tau_i||_\infty =  \inf\{t \ge 0 : \pr(\tau_i \leq t) = 1 \}, \quad i = 1, 2. 
	\end{equation}
	To arrive at a contradiction, we suppose that $\pr(\tau_i > 0) > 0, \, i = 1, 2$, and note that this implies $T_i > 0, \, i = 1, 2$. By Lemma \ref{boundary_stop_times} we know that there exist some boundaries $b_i: [0, \infty) \to [-\infty, \infty], \, i = 1, 2$, such that the stopping times $\tau_i, \, i = 1, 2$, are bounded from below by the first hitting times of these boundaries. 
	
	The key observation now is that, by Lemma \ref{devilish_lemma}, \textit{the processes $x_i + W^{(i)}_\cdot$ will become negative before hitting the boundaries $b_i(\cdot)$, with positive probability}. We will show that, on these events, each player will have to wait at least until his respective process hits the origin. And, as we have already seen in the proof of part \ref{result_3}, this will provide the other player enough time to obtain an arbitrarily large expected payoff. 
	
%	\begin{lemma}\label{devilish_lemma_}
%		Let $W_t, \, 0 \le t < \infty$, be a standard Brownian motion started at zero, consider the stopping time 
%		\begin{equation}\label{hitting_time_boundary_}
%		\tau \coloneqq \inf \left\{t \ge 0 : x + W_t \ge b(t)\right\}
%		\end{equation}
%		for an arbitrary measurable function $b: [0, \infty) \to [-\infty, \infty]$, and suppose $\pr(\tau > 0) > 0$.  Then $\pr(\tau > \sigma_y) > 0$ holds for every $y \in \mathbb{R}$, where
%		\begin{equation} \label{hitting_time_less_}
%		\sigma_y \coloneqq \inf\{t \ge 0 : x + W_t \leq y \}.
%		\end{equation}
%	\end{lemma}
%	Once this result has been proven, 
	We proceed as follows. Without loss of generality, suppose that we have $T_1 \leq T_2$ in \eqref{stopping_times_sups} and, for the sake of readability, we will drop our superscripts for the first player and define 
	\begin{equation*}
    	\tau \coloneqq \tau_1, \quad W \coloneqq W^{(1)}, \quad x \coloneqq x_1, \quad c(\cdot) \coloneqq c_1(\cdot)
	\end{equation*}
	along with
	\begin{equation*}
    	\widetilde{W}_t \coloneqq W_{\sigma_{-\delta} + t} - W_{\sigma_{-\delta}}, \, 0 \le t < \infty
	\end{equation*}
	where $\delta > 0$ and $\sigma_{-\delta}$ is defined as before as $
	\sigma_y \coloneqq \inf\{t \ge 0 : x + W_t \leq y \}$.
	Consider now the event $\{\tau > \sigma_{-\delta}\}$, which has positive probability by Lemma \ref{devilish_lemma}. We claim that on this event one must have 
	\begin{equation}\label{lower_bound_OST}
    	\tau \ge \sigma_{-\delta} +  \widetilde{\sigma}_{\delta}, \quad a.s.,
	\end{equation} 
	where $\widetilde{\sigma}_\delta$ is defined by $
	\widetilde{\sigma}_y \coloneqq \inf\{t \ge 0 : \widetilde{W}_t \leq y \}$. 
    Moreover, we claim that \eqref{lower_bound_OST} already contradicts the finiteness of the payoffs under $(\tau, \tau_2)$. 
    
    Indeed, we note the following chain of inequalities:
	\begin{equation}\label{boundary_ST_tail_bound}
    	\begin{split}
    	\pr(\tau > t) 
    	&\geq
    	\pr\left(\tau > \sigma_{-\delta}, \ \widetilde{\sigma}_\delta > t - \sigma_{-\delta}\right) \\
    	&=
    	\ex\left[\ex\left[\mathbf{1}_{\{\tau > \sigma_{-\delta}, \ \widetilde{\sigma}_\delta > t - \sigma_{-\delta}\}} \left| \mathcal{F}_{\sigma_{-\delta}} \right] \right. \right] \\
    	&=
    	\ex \left[\mathbf{1}_{\{ \tau > \sigma_{-\delta}\}} \pr\left(\widetilde{\sigma}_\delta > t-u\right) \Big|_{u = \sigma_{-\delta}} \right] \\
    	&\geq
    	\ex \left[\mathbf{1}_{\{ \tau > \sigma_{-\delta} \}} \pr(\widetilde{\sigma}_\delta > t) \right]\\
    	&=
    	\pr\left(\tau > \sigma_{-\delta}\right) \pr(\widetilde{\sigma}_\delta > t).
    	\end{split}
	\end{equation}  
	Here, the first inequality is a consequence of \eqref{lower_bound_OST}, the equalities in the second and third lines follow from the tower property and the independence of $\widetilde{W}_\cdot$ and $\mathcal{F}_{\sigma_{-\delta}}$, respectively; and the last inequality from the fact that $\sigma_{-\delta} \ge 0$. Finally, combining the tail bound \eqref{hitting_time_tail_bound}, the definition of the discount functions \eqref{def_discount_funct}, and the fact that $\pr(\tau>\sigma_{-\delta}) > 0$, we obtain
	\begin{equation*}
    	\liminf_{t \rightarrow \infty} \left(c_2(t) \sqrt{t} \right) \ge \liminf_{t \rightarrow \infty} \left(\pr(\tau > t) \sqrt{t}\right) > 0,
	\end{equation*}
	which implies that $(\tau, \tau_2)$ cannot be a Nash equilibrium, by Lemma \ref{unbounded_payoff_lemma}.
	
	Therefore, it only remains to prove the inequality \eqref{lower_bound_OST}. To obtain \eqref{lower_bound_OST}, we argue similarly to the proof of Part \ref{result_3}. 
    In particular, we recall the identity \eqref{prob_zero_intersection}, which was shown in Part \ref{result_3} to hold irrespective of the starting point $x$, and in the present notation can be written as \begin{equation}\label{prob_zero_intersection_part_4}
	\pr\left( \{x + W_{\tau} < 0\} \cap \{ c(\tau) > 0\}\right) = 0.
	\end{equation} 
 
	Firstly, we note that if $T_1 = \infty$, then \eqref{lower_bound_OST} follows immediately from \eqref{prob_zero_intersection_part_4}. Indeed, in this case the event $\{ c(\tau) > 0\}$ has full measure and on the event $$\{\sigma_{-\delta} < \tau < \sigma_{-\delta} + \widetilde{\sigma}_{\delta}\},$$ the process $x + W_\cdot$ is strictly negative by the definition of the respective stopping times.
    Secondly, the case $T_1 < \infty$ is impossible by the following reasoning. First, observe that  \begin{equation}\label{positivity_of_stopping_negative}
    	\begin{split}
    	\pr(x + W_\tau < 0) &\ge \pr\left( \sup_{0 \leq t \leq T_1}  B_t < 1, \, \sigma_{-1} < \tau \leq T_1 \right) \\
    	& = \pr\left(\sup_{0 \leq t \leq T_1}  B_t < 1 \right) \pr\big(\sigma_{-1} < \tau\big) \\
    	&> 0,
    	\end{split}
	\end{equation} where $\sigma_{-1}$ is again defined by $\sigma_{-1} \coloneqq \inf\{t \ge 0 : x + W_t \leq -1 \}$, $B_t \coloneqq W_{t + \sigma_{-1}} - W_{\sigma_{-1}}$ is a standard Brownian motion independent of $\mathcal{F}^{(1)}_{\sigma_{-1}}$, and we apply Lemma \ref{devilish_lemma} to obtain the last inequality.  This implies that $c(T_1) = 0$, since \eqref{prob_zero_intersection_part_4} is violated otherwise by arguments similar to those immediately after the inequality \eqref{negativity_at_stopping}. Therefore, $T_2 \leq T_1$, and, by symmetry, we obtain $T_1 = T_2 \coloneqq T$, along with the fact that $c(T) = c_2(T) = 0$.  However, by definition of $c(\cdot)$ in \eqref{def_discount_funct}, we then get $\tau \in [0, T)$ almost surely, and appealing once more to \eqref{prob_zero_intersection_part_4}, we obtain
	\begin{equation*}
    	\begin{split}
        	\pr(x + W_\tau < 0) &= \pr\left(x+W_\tau <0 , \, c(\tau) = 0\right) \\
        	&\leq \pr(\tau \ge T) \\
        	&= 0,
    	\end{split}
	\end{equation*} 
	which contradicts \eqref{positivity_of_stopping_negative}.
	
	\noindent \textbf{Proof of \ref{result_4}b:} 
    To show that a pair $(\tau_1, \tau_2)$ with $\, \pr(\tau_1 > 0) = 0$ and $\, \pr(\tau_2 > 0) = 0$ is a Nash equilibrium with finite expected payoffs if, and only if, $\min(x_1, x_2) \ge 0$, we argue as follows. Consider a pair of stopping times $(\tau_1, \tau_2) = (0, 0)$ and note that, by deploying such stopping times, the two players receive payoffs $x_1/2$ and $x_2/2$, respectively. However, if either of the players decides to deviate he will receive zero payoff due to Blumenthal's zero–one law. Therefore, if a player's initial position is negative, the player will always deviate; whereas if a player's initial position is non-negative, the player won't have the incentive to change his strategy, which proves the statement of \ref{result_4b}.
	
	\noindent \textbf{Proof of \ref{result_4}c:} 
    Finally, it remains to show that there exist Nash equilibria with finite expected payoffs in which $\, \pr(\tau_i > 0) > 0$ and $\, \pr(\tau_{3-i} > 0) = 0, \, i = 1, 2$, if, and only if, $\min(x_1, x_2) \le 0 < \max(x_1, x_2)$. 
    For the forward implication, we assume that $(\tau_1, \tau_2)$ is a Nash equilibrium in which, without loss of generality, $\tau_1 \equiv 0$ and $\pr(\tau_2 > 0) > 0$. Arguing as in the proof of part  \ref{result_4b} we immediately obtain that the initial position of the first player must be non-negative, i.e., $x_1 \ge 0$, while the initial position of the second player must be non-positive, i.e., $x_2 \le 0$. To complete the proof of the forward implication, we just note that $x_1 = 0$ provides a contradiction, since the first player can deviate and obtain a strictly greater expected payoff by playing $\widetilde{\tau}_1 = \tau_\varepsilon \coloneqq \inf \{t \ge 0: W^{(1)}_t = \varepsilon\}$ for some $\varepsilon > 0$. Indeed, by choosing such a stopping time the first player's expected payoff is strictly positive since  
	\begin{equation*}
    	\begin{split}
        	J_1(\widetilde{\tau}_1, \tau_2, x_1, x_2) 
        	&=
        	\ex \left[ f\left(W^{(1)}_{\tau_\eps}  \right) c_1(\tau_\eps)\right]
        	\\ &=
        	\ex \big[ f(\eps) \, c_1(\tau_\eps)\big]
        	\\ &\ge
        	f(\eps)\, \pr(\tau_2 > \tau_\eps)
        	\\ &
        	> 0,
    	\end{split}
	\end{equation*}
	where the last inequality follows from the fact that $f(\eps) > 0$ from the linearity of $f(\cdot)$, while $\pr(\tau_2 >\nobreak \tau_\eps) > 0$ by independence of $\tau_2$ and $\tau_\eps$ and the fact that $\pr(\tau_2 > 0) > 0$. Therefore, $(\tau_1, \tau_2)$ with $\tau_1 \equiv 0$, $\pr(\tau_2 >\nobreak 0) > 0$ is not a Nash equilibrium when $x_1 = 0 \ge x_2$.
	
	For the reverse implication, we suppose, again without loss of generality, that $x_1 > 0 \ge x_2$. We claim that any stopping time $\tau_2$ with a survival function that decays fast enough then provides a Nash equilibrium $(\tau_1, \tau_2)$ with $\tau_1 \equiv 0$. It is clear that the second player does not have an incentive to deviate, since the immediate stopping of the first player will not allow him to obtain any positive payoff. As for the first player, intuitively speaking, he will not deviate because the discount function in his respective optimal stopping problem decreases so rapidly that he will not be able to gain anything by continuing the game instead of stopping immediately. 
	
	To make this intuition precise, we proceed as follows. We use a result of Shepp \cite{Shepp}, who considers the optimal stopping problem 
    of maximizing the expected payoff
    \begin{equation*}
        \ex \left[ \frac{u + W_\tau}{b + \tau} \right], \quad u \in \mathbb{R}, \, b > 0,
    \end{equation*}
    i.e., a problem similar to \eqref{value_func_as_discounting} but with reward function $u + x, \, u \in \mathbb{R}$ and discount function $1/(b + t), \, b > 0$. Shepp proves that in this problem
    % \eqref{value_func_in_unbounded_lem} with reward function $u + x, \, u \in \mathbb{R}$ and discount function $1/(b + t), \, b > 0$, 
    it is optimal to stop immediately if $u \ge \alpha \, \sqrt{b}$, where $\alpha$ is the unique real root of 
	\begin{equation*}
	   \alpha=\left(1-\alpha^{2}\right) \int_{0}^{\infty} e^{\lambda \alpha-\lambda^{2} / 2} d \lambda.
	\end{equation*}
	To make use of this result in our setting, consider any stopping time $\tau_2$ with survival function $h(\cdot) = \pr(\tau_2 > \cdot)$ such that
	\begin{equation}\label{upper_bound_survival_function}
	   \pr(\tau_2 > t) = h(t) \le \frac{b}{b + t}, \quad t \ge 0,
	\end{equation}
	where $b > 0$ satisfies $\alpha \sqrt{b} \le x_1$. The existence of such stopping times follows, for example, from Theorem 1 in Anulova \cite{Anulova}, which states that for any distribution $\mu$ on $(0, \infty]$, there exists a boundary such that the first time a reflected Brownian motion hits this boundary has distribution $\mu$. Then, if the second player chooses such a stopping time $\tau_2$, the first player's optimal response is $\tau_1 \equiv 0$, since for any other stopping time $\tau$ his expected payoff can be bounded above by
	\begin{equation*}
    	\begin{split}
    	J_1(\tau, \tau_2, x_1, x_2) 
    	&=
    	\ex \left[f\left(x_1 + W^{(1)}_{\tau} \right)  h(\tau) \right]
    	\\ &\le
    	bk \, \ex \left[ \frac{x_1 + W^{(1)}_\tau}{b + \tau}\right]
    	\\ &\le
    	bk \frac{x_1}{b}
    	\\ &=
    	k \, x_1 = J_1(\tau_1, \tau_2, x_1, x_2),
    	\end{split}
	\end{equation*}
	where the first inequality follows from the choice of $h(\cdot)$ and the second inequality follows from the choice of $b$ and the results of \cite{Shepp} just mentioned. Therefore, the first player will not deviate, which establishes that $(\tau_1, \tau_2)$ is a Nash equilibrium. Since there are infinitely many stopping times that satisfy the condition \eqref{upper_bound_survival_function}, we can obtain infinitely many Nash equilibria of such type. This completes the proof of part \ref{result_4c} and the proof of Theorem \ref{main_thm}. \hfill \qed {\parfillskip0pt\par}

\section{Ramifications and Extensions}\label{sec_extensions}

In this section, we illustrate some immediate ramifications and a corollary to Theorem \ref{main_thm}, which classifies all Nash equilibria in the problem of maximizing \eqref{value_func_as_discounting} in the special case where $f(\cdot)$ is affine. Subsequently, we discuss three extensions of our results.

\subsection{Ramifications of Theorem \protect\ref{main_thm}}

We first show that some simple relaxations to the conditions \ref{assm_1} and \ref{assm_4} may be made, given in Remark \ref{relaxation_remark} below. In the special case of an affine reward function $f(\cdot)$, these relaxations, together with Theorem \ref{main_thm}, enable us to solve the Dynkin game under consideration completely in the sense that we are able to classify all possible Nash equilibria. This result is given in Corollary \ref{completely_solved_game} below.

\begin{remark}\label{relaxation_remark}
	Note that assumptions \ref{assm_1} and \ref{assm_4} may be replaced by the following: 
	\begin{enumerate}[leftmargin=45pt, label = \textbf{(A\arabic*)$^{\prime}$}]
		\item There exists some $m \in \mathbb{R}$ such that $f(x) < 0$ for $x < m$ and $f(x) \ge 0$ for $x \ge m$, \label{assm_1_prime}
		\addtocounter{enumi}{2}
		\item $f(x) = kx + b$ with $k > 0$, $b \in \mathbb{R}$, \label{assm_4_prime}
	\end{enumerate}
	with the appropriate changes $\max(x_1, x_2) < m$ in \ref{result_3}, $\min(x_1, x_2) \ge -b/k$ in \ref{result_4b} and $\max(x_1, x_2) > -b/k$, $\min(x_1, x_2) \leq -b/k$ in \ref{result_4c} of Theorem \ref{main_thm}.
	
	Indeed, for a function $f(\cdot)$ satisfying \ref{assm_1_prime}, we note that $g(x) \coloneqq f(x + m)$ satisfies \ref{assm_1} and further satisfies \ref{assm_2} and \ref{assm_3} whenever $f(\cdot)$ does.  Noting that in terms of $g(\cdot)$ the players' potential rewards are given by $g(x_i - m + W_{\tau_i}^{(i)})$, $i \in \{1, 2\}$, the optimal stopping problem in our setting involving $f(\cdot)$ and starting points $(x_1, x_2)$ falls under the purview of Theorem \ref{main_thm}, as it is equivalent to the problem involving reward function $g(\cdot)$ and starting points $(x_1 - m, x_2 - m)$. Similarly, if $f(\cdot)$ satisfies \ref{assm_4_prime}, the optimal stopping problem with reward function $f(\cdot)$ and starting points $(x_1, x_2)$ is equivalent to the problem with reward function $g(x) \coloneqq kx$ and starting points $(x_1 + b/ k, x_2 + b/k)$, to which part \ref{result_4} of Theorem \ref{main_thm} may be applied.
\end{remark}

\begin{corollary}\label{completely_solved_game}
	Consider the two-player game described in Section \ref{model}, started at $(x_1, x_2) \in \mathbb{R}^2$. Assume that $f(\cdot)$ is affine and strictly increasing, i.e., $f(x) = kx + b$, where $k > 0, b \in \mathbb{R}$. Then there exist Nash equilibria $(\tau_1, \tau_2)$ for such a game, which can be characterized as follows:
	\begin{enumerate}
		\item Either both players receive infinite expected payoffs by playing $\tau_1$ and $\tau_2$, respectively, and there exist infinitely many such equilibria;
		
		\item Or both players receive finite payoffs. This is possible only in the following cases:
		\begin{enumerate}
			\item $\min(x_1, x_2) \ge -b/k$, and the equilibrium is given by $(\tau_1, \tau_2) = (0, 0)$ with respective payoffs $(x_1/2, x_2/2)$, 
			\item $\min(x_1, x_2) \le -b/k, \max(x_1, x_2) > -b/k$, and the game terminates immediately, since the player with $x_i > 0$ chooses the stopping time $\tau_i \equiv 0$. There are infinitely many Nash equilibria in this case, but the players' expected payoffs are always given by $(\max(x_1, x_2), 0)$ or  $(0, \max(x_1, x_2))$ depending on whether $x_1 > x_2$ or $x_1 < x_2$ respectively.
		\end{enumerate} 
	\end{enumerate}
\end{corollary}

We now discuss several extensions of our results.

\subsection{Different Reward Functions}

First, we note that the reward function $f(\cdot)$ of the Dynkin game under consideration need not be the same for both players. More precisely, if the expected payoffs \eqref{value_functions_i} of the players depend on different reward functions $f_i: \mathbb{R} \to \mathbb{R}, \, i = 1, 2$, i.e., if the expected payoffs are given by 
\begin{equation*}
    J_i(\tau_i, \tau_{3-i}, x_1, x_{2}) \coloneqq \ex \left[f_i\left(x_i + W^{(i)}_{\tau_i}\right) \left( \mathbf{1}_{\{\tau_i < \tau_{3-i}\}} + \frac{1}{2}\mathbf{1}_{\{\tau_i = \tau_{3-i} < \infty\}}\right)\right],
\end{equation*}
then all statements of Theorem \ref{main_thm} still hold provided that both the reward functions $f_i(\cdot), \, i = 1, 2$ satisfy the assumptions \ref{assm_1}--\ref{assm_4} of Section \ref{model}. This follows from the fact that all arguments in the proof of Theorem \ref{main_thm} only rely on the properties of the reward function and do not use the fact that the reward function is the same for both players.

\subsection{$n$--Player Games}

Another extension is to note that the statement \ref{result_1} of Theorem \ref{main_thm} holds if we consider an $n$-player game instead of a 2-player game, as discussed previously. Indeed, with some necessary yet obvious adjustments of notation, it is clear that in an $n$-player game, each player faces an $n$-player analogue of the optimal stopping problem \eqref{value_func_as_discounting}. More precisely, denoting by $\vec{x} \coloneqq(x_1, \dots, x_n)$ the vector of initial positions of all players and by $\vec{\tau} \coloneqq (\tau_1, \dots, \tau_n)$ the vector of stopping times chosen by the players, and arguing similarly to the beginning of the proof of Theorem \ref{main_thm}, one can show that each player $i$ needs to maximize the expected payoff given by the expression 
\begin{equation}\label{value_func_as_discounting_n_players}
    J_i(\vec{\tau}, \vec{x}) = \ex \left[ f\left(x_i + W_{\tau_{i}}^{(i)} \right) c_i(\tau_i)\right], \quad i = 1, \dots, n, 
\end{equation}
with 
\begin{equation}\label{def_discount_funct_n_players}
    c_i(t) \coloneqq \prod_{\substack{j=1 \\ j \neq i}}^{n} \pr(\tau_{j} > t) + d_i(t), \quad 0 \le t < \infty,
\end{equation}
where $d_i(t) \ge 0$ is a term responsible for the winning in case of ties between players (we will not need its explicit form, it is enough to assume it is non-negative). 

However, using \eqref{value_func_as_discounting_n_players} and proceeding exactly as in the proof of Part \ref{result_1} of Theorem \ref{main_thm}, one can show that the vectors of stopping times $(\tau_a^{(1)}, \dots, \tau_a^{(n)})$, with $\tau_a^{(i)}, \, i = 1, \dots, n$ defined as in \eqref{square_root_st_times}, will again constitute Nash equilibria with infinite expected payoffs, provided that $a$ is large enough so that the quantity $\beta(a)$ defined in \eqref{breiman_conditions} satisfies 
\begin{equation}\label{upper_bound_on_beta}
    \beta(a) < \frac{\gamma}{2n}.
\end{equation} 
Indeed, in this case, the analogue of the expression \eqref{payoff_under_root_coundary} for the $n-$player game has a power parameter 
\begin{equation*}
    q \coloneqq \frac{\gamma}{2} - (n-1)\beta(a),
\end{equation*}
and thus the lower bound \eqref{infinite_lower_bound} is still valid, as the analogous $r(a)$ satisfies
\begin{equation*}
    r(a) \coloneqq \frac{2\beta(a)}{\gamma - 2(n-1)\beta(a)} < 1,
\end{equation*}
as consequence of \eqref{upper_bound_on_beta}.

\subsection{Mixed Strategies}

Finally, we show that all statements of Theorem \ref{main_thm} continue to hold if players are allowed to use mixed strategies. In other words, strategy randomization does not create categorically different equilibria in the Dynkin game under consideration. 
For brevity, and similarly to the previous subsections, we only provide the main ideas, leaving rigorous justifications to the reader. 

We use the notion of mixed strategies from Touzi and Vieille \cite{TouVie}. Namely, we extend the setting of Section \ref{model} by enlarging the probability space from $(\Omega, \mathcal{F}, \pr)$ to 
$\big([0, 1] \times [0, 1] \times \Omega, \mathcal{B}([0, 1]) \otimes \mathcal{B}([0, 1]) \otimes \mathcal{F}, \lambda_1 \otimes \lambda_2 \otimes \pr\big) \eqqcolon (\overline{\Omega}, \overline{\mathcal{F}}, \overline{\pr})$, 
where $\mathcal{B}([0, 1])$ is the Borel $\sigma$--algebra on $[0, 1]$ and $\lambda_i, \, i = 1, 2$, is the Lebesgue measure.
A mixed strategy for player $i, \, i = 1, 2,$ is then defined as a $\lambda_i \times \pr$--measurable function $\phi_i: [0, 1] \times \Omega \to [0, T], \, i = 1, 2$, such that
\begin{equation}
    \text{for } \ \lambda_i-\text{a.e.} \ \ u \in [0, 1], \quad \omega \mapsto \phi_i(u, \omega) \text{ is an } \mathbb{F}^{(i)}-\text{stopping time}.
\end{equation}
We denote the corresponding sets of mixed strategies by $\Phi_i, \, i = 1,2$, and denote $\tau_i^u \coloneqq \phi_i(u, \cdot)$.

When the players choose their strategies $\phi_1$ and $\phi_2$, respectively, they receive payoffs 
\begin{equation}
    R_i(u, v, x_1, x_2) \coloneqq f\left(x_i + W^{(i)}_{\tau_i^u}\right) \left( \mathbf{1}_{\{\tau_i^u < \tau_{3-i}^v\}} + \frac{1}{2}\mathbf{1}_{\{\tau_i^u = \tau_{3-i}^v < \infty\}}\right), \quad i = 1, 2,
\end{equation}
and maximize their expected payoffs
\begin{equation}
    \begin{split}
        \overline{J}_i(\phi_i, \phi_{3-i}, x_1, x_2) 
        &\coloneqq 
        \overline{\ex}\big[R_i \big],
        \\
        \overline{V}_i(\phi_{3-i}, x_1, x_2) 
        &\coloneqq 
        \sup\limits_{\phi \in \Phi_i} \overline{J}_i(\phi, \phi_{3-i}, x_1, x_2).
    \end{split}
\end{equation}
Here, $\overline{\ex}[ \, \cdot \, ]$ denotes expectation with respect to the new measure $\overline{\pr}$.
\\

We are ready to show that all statements of Theorem \ref{main_thm}, appropriately adjusted for the mixed strategies $\phi_i, \, i = 1, 2$, continue to hold in the mixed strategy setting. Moreover, all arguments hereafter repeat the proofs from Section \ref{sec_proof_of_main_thm} with minor modifications.

The statement \ref{result_1} holds trivially since the set of pure strategies is a subset of the set of mixed strategies; thus, the proof of the statement \ref{result_1} from Section \ref{subsec_proof_1} remains valid.

To obtain the statement \ref{result_2}, we first rewrite the players' expected payoffs as
\begin{equation}\label{integral_representation_of_payoffs}
    \begin{split}
        \overline{J}_i(\phi_i, \phi_{3-i}, x_1, x_2) 
        &= 
        \overline{\ex}\left[R_i \right]
        \\&=
        \int_0^1 \int_0^1 \int_\Omega R_i(u, v, x_1, x_2) \, d\pr \, du \, dv
        \\&=
        \int_0^1 \int_0^1 \ex\left[
            f\left(x_i + W^{(i)}_{\tau_i^u}\right)
            c_i\big(v, \tau_i^u\big) 
        \right] du \, dv
        \\&=
        \int_0^1 \ex\left[
            f\left(x_i + W^{(i)}_{\tau_i^u}\right) c_i\big(\tau_i^u \big)
        \right] du.
    \end{split}
\end{equation}
Here, the second equality follows from the tower property as in \eqref{tower_property}, and similarly to \eqref{def_discount_funct}, we denote, for $i = 1, 2$,
\begin{equation} \label{c_dfn_mixed_strat}
    \begin{split}
        c_i(v, t) &\coloneqq \pr\left(\tau_{3-i}^v > t\right) + \frac{1}{2} \pr\left(t = \tau_{3-i}^v < \infty\right), \quad 0 \le t < \infty,
        \\
        c_i(t) &\coloneqq \int_0^1 c_i(v, t) \, dv, \quad 0 \le t < \infty.
    \end{split}
\end{equation}

We argue now by contradiction, similarly to the proof of the statement \ref{result_2} from Section \ref{subsec_proof_2}. We assume, without loss of generality, that the first player receives an infinite expected payoff, meaning that the expressions in \eqref{integral_representation_of_payoffs} with $i = 1$ are infinite. Then, there exists a set $A \in \mathcal{B}([0, 1])$ with positive Lebesgue measure such that
\begin{equation*}
    \ex\left[
        f\left(x_1 + W^{(1)}_{\tau_1^u}\right) c_1\big(\tau_1^u \big) 
    \right] \ge k\Big(1 + 2|b|\Big)
\end{equation*}
for every $u \in A$, where the constants $k$ and $b$ are given in the assumption \ref{assm_3}. 
However, using \ref{assm_3} (equivalently, taking the difference of the first inequalities in \eqref{upperboundminus} and \eqref{lowerboundplus}), we obtain
\begin{equation*}
    \ex\left[x_1+W^{(1)}_{\tau_1^u}\right] \ge 1, \quad \forall \, u \in A.
\end{equation*}
Using the result of Novikov \cite{Novikov} once more, we obtain
\begin{equation*}
    \liminf_{t \rightarrow \infty} \left(\pr(\tau_1^u > t) \sqrt{t}\right) \ge 1, \quad \forall \, u \in A,
\end{equation*}
which implies that
\begin{equation*}
    \liminf_{t \rightarrow \infty} \left(c_2(t)\sqrt{t \log \log (t)}\right) \ge \liminf_{t \rightarrow \infty} \int_A \pr(\tau_1^u > t) \sqrt{t \log \log (t)} \, du = \infty. 
\end{equation*}
An appeal to Lemma \ref{unbounded_payoff_lemma} then yields that player 2 may obtain an arbitrarily large expected payoff with deterministic strategies, which is the desired contradiction.

To obtain the statement \ref{result_3}, we proceed again by contradiction, assuming that $(\phi_1, \phi_2)$ is an equilibrium where both players obtain finite expected payoffs. We first note that the argument of \eqref{comparing_payoffs}, which implies \eqref{prob_zero_intersection}, applies for Lebesgue almost every $(u,v) \in [0,1]^2$; specifically, we have
\begin{equation*}
    \pr\Big(\{x_i + W^{(i)}_{\tau_i^u} < 0 \} \cap \{c_i(v,\tau_i^u) >0\}\Big) = 0,  \ \ \text{for} \ \ \lambda_1 \otimes  \lambda_2-\text{a.e.} \ \ (u, v) \in [0, 1]^2.
\end{equation*}
% for Lebesgue almost every $(u,v) \in [0,1]^2$. 
Therefore, the remaining steps of our proof of the statement \ref{result_3} from Section \ref{subsec_proof_3} hold for all such tuples $(u,v)$. 
In particular, with 
$\sigma^{(i)}_\xi = \nobreak \inf \{t \ge\nobreak 0 : W_t^{(i)} = \xi\}, \, \xi \in \mathbb{R}$,
similarly to \eqref{surv_func_lower_bound}, we have
\begin{equation*}
    c_{3-i}(u,t) \ge\pr\left(\tau_i^u > t\right) \ge \pr\left(\sigma_{-x_i}^{(i)} > t\right), \quad i = 1, 2.
\end{equation*}
Thus, another appeal to Lemma \ref{unbounded_payoff_lemma} allows player $3 - i, \, i = 1, 2$, to obtain an arbitrarily large expected reward with deterministic strategies, which is the desired contradiction. 

To obtain \ref{result_4a}, we proceed by contradiction once more, assuming that $(\phi_1, \phi_2)$ is an equilibrium where both players obtain finite expected payoffs. We start by observing that, for any fixed mixed strategy $\phi_{3-i}$ for player $3-i, \, i = 1, 2$, we have
\begin{equation*}
    \sup_{\phi \in \Phi_i} \overline\ex[R_i] = \sup_{\tau \in \mathcal{T}_i} \ex\left[f\left(x_i + W^{(i)}_\tau\right) c_i(\tau) \right],
\end{equation*}
where $c_i(\cdot)$ is defined in \eqref{c_dfn_mixed_strat}. The ``$\leq$'' direction is implied immediately by the last equality in \eqref{integral_representation_of_payoffs}. The reverse inequality follows trivially from the fact that deterministic strategies are a subset of mixed strategies. Therefore, using the same arguments as in the proof of the statement \ref{result_4a} in Section \ref{subsec_proof_4}, we obtain that $\tau_i^u$ is bounded from below by the first hitting time of a boundary for Lebesgue almost every $u \in [0,1]$. This once again provides player $3-i$ the opportunity to obtain an arbitrarily large expected payoff with deterministic strategies by Lemma \ref{unbounded_payoff_lemma}, which is the desired contradiction.

The statement \ref{result_4b} follows from the identical argument to the corresponding proof in Theorem \ref{main_thm}, replacing $\pr$ by $\overline \pr$.

Finally, the statement \ref{result_4c} also follows immediately from the corresponding proof in Theorem \ref{main_thm} since the set of deterministic strategies is a subset of the set of mixed strategies.

\section{Technical results}\label{sec_technical_proofs} 

In this section, we provide proofs of the technical Lemmas \ref{unbounded_payoff_lemma}, \ref{value_func_in_unbounded_lem}, and \ref{devilish_lemma} that were deployed in the proof of Theorem \ref{main_thm}; and also, for the sake of completeness, state the result by Breiman \cite{Breiman}, which was used in the proof of Part \ref{result_1} of Theorem \ref{main_thm}. We do this in the order in which the corresponding results appear in the paper.

\begin{theorem}[Breiman (1966)] \label{breiman_result}
    Suppose $W = \{W_t: t \ge 0\}$ is a standard Brownian motion. Define
    \begin{equation*}
        \rho_a \coloneqq \inf\left\{t \ge 1 : |W_t| \ge a \sqrt{t} \right\}.
    \end{equation*}
    Then, there exists a function $\beta: (0,\infty) \rightarrow (0,\infty)$ and a constant $c > 0$ such that
    \begin{equation}\label{breiman_limit}
        \lim_{t \rightarrow \infty} \frac{\pr\big(\rho_a > t \, | \, W_1 = 0\big)}{t^{-\beta(a)}} = c \quad \text{ and } \quad \lim_{a \rightarrow \infty} \beta(a) = 0.
    \end{equation}
\end{theorem}

This result is taken from Theorem 1 of Breiman \cite{Breiman}, and we refer the reader there for its proof. Note that $\rho_a$ in Theorem \ref{breiman_result} above satisfies the inequality
\begin{equation*}
    \rho_a \le 1 + \inf\{s \ge 0 : W_{1+s} - W_1 \ge a\sqrt{s+1} - W_1\}.
\end{equation*}
Therefore, the statement \eqref{breiman_limit} implies that
\begin{equation*}
    \liminf_{t \rightarrow \infty} \frac{\pr\big(\widetilde{\rho}_a > t-1\big)}{t^{\beta(a)}} \ge c,
\end{equation*}
where 
\begin{equation*}
    \widetilde{\rho}_a \coloneqq \inf\left\{t \ge 0 : W_t \ge a\sqrt{t+1}\right\}.
\end{equation*}
Since $\widetilde{\rho}_a$ corresponds precisely to our definition \eqref{stopping_time_tilde} of $\widetilde{\tau}_a^{(i)}, \, i= 1, 2$, in the proof of Part \ref{result_1} of Theorem \ref{main_thm}, this justifies \eqref{breiman_conditions}.

We are now ready to state and prove our three technical Lemmas.

\begin{lemma}\label{unbounded_payoff_lemma}
    Fix $x \in \mathbb{R}$.  Consider a standard Brownian motion $W = \{W_t: t \geq 0\}$, let $\, \mathbb{F}^W$ denote the natural filtration of $\, W$, and let $\mathcal{T}$ denote the set of all  $\, \mathbb{F}^W$ -- stopping times for which $\ex \left[ f \left(x + W_{\tau}\right) \right]$ is well defined. 
    
    Consider the optimal stopping problem for the Brownian motion $W$, with reward function $f(\cdot)$ and discount function $c(\cdot)$, namely
    \begin{equation}\label{value_func_in_unbounded_lem}
        V(x) \coloneqq \sup\limits_{\tau \in \mathcal{T}} \ex \Big[f(x + W_\tau) \, c(\tau)\,\mathbf{1}_{\{\tau < \infty\}}\Big].
    \end{equation}
    If the function $f(\cdot)$ satisfies assumption \ref{assm_2} of Theorem \ref{main_thm} with $\gamma \ge 1$, and if
    \begin{equation}\label{condition_on_disc_func}
        \liminf_{t \rightarrow \infty} \left(c(t) \sqrt{t \log \log (t)}\right) = \infty,
    \end{equation}
    then $V(x) = \infty, \ \forall \ x \in \mathbb{R}.$
\end{lemma}

\begin{proof}
    By \ref{assm_2} with $\gamma \ge 1$, there exist constants $\beta > 0$ and $K \in \mathbb{R}$ such that, for every $y \in [K, \infty)$, we have $f(y) \ge \beta y$.
    Fix $M > 0$.  Condition \eqref{condition_on_disc_func} implies that there exists some $T \ge 0$ such that for all $t \ge T$ we have \begin{equation}\label{unbounded_payoff_lemma_bounds}
        c(t) \sqrt{t \log \log (t)} \ge \frac{M}{\beta} \qquad \text{and} \qquad \sqrt{t \log \log(t)} \ge K.
    \end{equation}  
    Consider the stopping time
    \begin{equation*}
        \rho_T \coloneqq \inf \left\{t \ge T : x + W_t \ge \sqrt{t \log \log (t)} \right\},
    \end{equation*}
    which is almost surely finite by the Law of the Iterated Logarithm. Using the bounds \eqref{unbounded_payoff_lemma_bounds} we obtain
    \begin{align*}
        V(x) =
        \sup\limits_{\tau \in \mathcal{T}} \ex \left[f(x + W_\tau) \, c(\tau) \, \mathbf{1}_{\{ \tau < \infty\}} \right] 
        &\ge
        \ex \big[f(x + W_{\rho_T}) \cdot c(\rho_T) \big] \\
        &=
        \ex \left[f\left(\sqrt{\rho_T \log \log \rho_T} \right)  \cdot c(\rho_T) \right] \\
        &\ge
        \ex \left[\beta \sqrt{\rho_T \log \log \rho_T} \cdot c(\rho_T) \right] \\ 
        &\ge
        M,
    \end{align*} 
    and since $M$ was arbitrary, this completes the proof.
\end{proof}

\begin{lemma}\label{boundary_stop_times}
	Consider the optimal stopping problem of \eqref{value_func_in_unbounded_lem} with reward function $f(\cdot)$, discount function $c(\cdot)$ and starting position $x \in \mathbb{R}$. Assume that $f(\cdot)$ satisfies assumption \ref{assm_4} and that the value of the optimal stopping problem is finite. If  $\tau^*$ is an optimal stopping time for this problem, then 
    \begin{equation}\label{lower_bound_by_boundary}
        \tau^* \ge \inf \left\{t \ge 0 : x + W_t \ge b(t) \right\},
    \end{equation}
	where $b: [0, \infty) \to [-\infty, \infty]$ is some boundary.
\end{lemma}

\begin{proof}
	The general theory of optimal stopping (see, e.g., Chapter I.2.2 in \cite{PesShi}, in particular, the proof of Theorem 2.4) implies that any optimal stopping time $\tau^*$ of the problem \eqref{value_func_in_unbounded_lem} satisfies
     \begin{equation}\label{lower_bound_hittin_stopping_region}
         \tau^* \ge \inf\{t: (t, x + W_t) \in \mathcal{S}\} \quad a.s.,
     \end{equation}
     where 
     \begin{equation*}
         \mathcal{S} \coloneqq \{(t,x) \in \mathbb{R}_+ \times \mathbb{R} : c(t)f(x) = V(t,x)\}
     \end{equation*}
     is the so-called \textit{stopping region}, 
    and $V(\cdot,\cdot)$ is the value function for the problem where one starts discounting at time $t$; that is,
    \begin{equation*}
        V(t,x) \coloneqq \sup_{\tau \in \mathcal{T}} \ex \left[ f(x + W_\tau) \, c(t+\tau) \,\mathbf{1}_{\{\tau < \infty\}} \right].
    \end{equation*}
    It suffices to show that, if at some point $(t,x)$ it is better to stop, i.e., that
	\begin{equation}\label{stopping_condition}
	   k \, x \,c(t) \ge \ex \left[k \, (x + W_{t+\tau}) \, c(t + \tau) \, \Big| \, W_t = 0 \right]
	\end{equation} 
	holds for every stopping time $\tau \in \mathcal{T}$, then for any $y > x$ the point $(t,y)$ will also satisfy \eqref{stopping_condition}. Indeed, if this is true, then the stopping region $\mathcal{S}$ of the problem can be represented as
    \begin{equation}\label{stopping_region}
        \mathcal{S} = \{(x, t): x \ge b(t)\},
    \end{equation}
    where $b(\cdot)$ is simply the function defined by
	\begin{equation*}
    	b(t) \coloneqq \inf \{x : (x, t) \text{ satisfies \eqref{stopping_condition}} \}.
	\end{equation*}  
    As a result, \eqref{lower_bound_by_boundary} will follow immediately from \eqref{stopping_region} and \eqref{lower_bound_hittin_stopping_region}.
	
    Now, note that the right-hand side of \eqref{stopping_condition} can be written as 
	\begin{equation}\label{stopping_condition_2}
    	\ex \left[k \, (x+W_{t + \tau})\, c(t + \tau) \, \Big| \, W_{t} = 0 \right]= \ex \Big[k \, (x + B_\tau)\, c(t + \tau)  \Big],
	\end{equation}
	where $B_s \coloneqq W_{t + s} - W_{t}, \, s \ge 0$ is a standard Brownian motion, independent of $\mathbb{F}_t$. It remains to observe that \eqref{stopping_condition_2} implies that $(y, t)$ also belongs to the stopping region by virtue of the following chain of inequalities:
	\begin{equation*}
    	\begin{split}
    	k \, y \, c(t) &= 
    	k \, x\,c(t) + k \, (y-x)\,c(t) \\
    	&\ge
    	\ex \Big[k \, (x + B_\tau)\,c(t + \tau)  \Big] + k \, (y-x)\,c(t)\\
    	&=\ex \Big[k \, (y + B_\tau)\,c(t + \tau)  \Big] + \ex \Big[k \, (x -y)\,c(t + \tau)  \Big]+ k \, (y-x)\,c(t)\\
    	&= \ex \Big[k \, (y + B_\tau)\,c(t + \tau)  \Big] + \ex \Big[k \, (y-x)\,(c(t)-c(t + \tau))  \Big]\\
    	&\ge \ex \Big[k \, (y + B_\tau)\,c(t + \tau)  \Big],
    	\end{split}
	\end{equation*}
	where the first inequality follows from the choice of $(x, t)$ and \eqref{stopping_condition}, \eqref{stopping_condition_2}, whereas the second inequality follows since $y>x$ and $c(\cdot)$ is a non-increasing function. Using the equivalence of \eqref{stopping_condition} and \eqref{stopping_condition_2}, and the fact that $\tau$ was chosen arbitrarily, concludes the proof.
\end{proof}

The proof of our last auxiliary result Lemma \ref{devilish_lemma} below relies on a generalized Fortuin-Kasteleyn-Ginibre (FKG) inequality developed by Barbato in \cite{FKG_inequality}, which we state here after introducing the necessary notions. Assume that $(\Omega, \mathcal{F}, \pr)$ is the canonical Wiener space with $(\Omega_T, \mathcal{F}_T, \pr)$ being its restriction to a finite time interval $[0,T]$ for fixed $T > 0$. This will be a standing assumption for the remainder of this section and is made without any loss of generality. We define the following partial ordering on $\Omega_T$: for $\omega_1, \omega_2 \in \Omega_T$, we say $\omega_1 \leq \omega_2$ if the inequality
    \begin{equation*}
        W_{t_2}(\omega_1) - W_{t_1}(\omega_1) \leq W_{t_2}(\omega_2) - W_{t_1}(\omega_2)
    \end{equation*}
    holds for all $0 \leq t_1 \leq t_2 \leq T$. A set $A \in \mathcal{F}_T$ is said to be \textit{increasing} if $\omega_1 \in A$ and $\omega_1 \leq \omega_2$ implies that $\omega_2 \in A$.  The following result is the so-called FKG inequality for Brownian motion. For its proof, see Theorem 4 of Barbato \cite{FKG_inequality}.
    
\begin{theorem}[Barbato (2005)] \label{FKG_inequality}
    If $A$ and $B$ are increasing sets contained in the sigma field $\mathcal{F}_T$, $T \ge 0$, we have 
    \begin{equation} \label{FKG_statement}
        \pr(A \cap B) \ge \pr(A) \, \pr(B).
    \end{equation} 
\end{theorem}

We are now ready to state and prove Lemma \ref{devilish_lemma}.

\begin{lemma}\label{devilish_lemma}
	Let $W_t, \, 0 \le t < \infty$, be a standard Brownian motion started at zero, consider the stopping time 
	\begin{equation}\label{hitting_time_boundary}
    	\tau \coloneqq \inf \left\{t \ge 0 : x + W_t \ge b(t)\right\}
	\end{equation}
	for an arbitrary measurable function $b: [0, \infty) \to [-\infty, \infty]$, and suppose $\pr(\tau > 0) > 0$.  Then $\pr(\tau > \sigma_y) > 0$ holds for every $y \in \mathbb{R}$, where
	\begin{equation} \label{hitting_time_less}
    	\sigma_y \coloneqq \inf\{t \ge 0 : x + W_t \leq y \}.
	\end{equation}
\end{lemma}
\begin{proof}
    If $y \ge x$, then $\sigma_y \equiv 0$ and the conclusion is trivial.  For $y < x$, we will apply Theorem \ref{FKG_inequality}. 
    %First, let us introduce the necessary notions, then state the FKG inequality.
    %Without loss of generality, we assume that $(\Omega, \mathcal{F}, \pr)$ is the canonical Wiener space with $(\Omega_T, \mathcal{F}_T, \pr)$ the restriction to a time interval $[0,T]$ for fixed $T > 0$. 
    % We define the following partial ordering on $\Omega_T$: for $\omega_1, \omega_2 \in \Omega_T$, we say $\omega_1 \leq \omega_2$ if 
    % \begin{equation}
    % W_{t_2}(\omega_1) - W_{t_1}(\omega_1) \leq W_{t_2}(\omega_2) - W_{t_1}(\omega_2)
    % \end{equation}
    % holds for all $0 \leq t_1 \leq t_2 \leq T$. A set $A \in \mathcal{F}_T$ is said to be \textit{increasing} if $\omega_1 \in A$ and $\omega_1 \leq \omega_2$ implies that $\omega_2 \in A$.  Now, the FKG inequality for Brownian motion is Theorem 4 of \cite{FKG_inequality}, which states that for increasing sets $A,B \in \mathcal{F}_T$, we have 
    % \begin{equation} \label{FKG_statement}
    % P(A \cap B) \ge P(A) P(B).
    % \end{equation}  
    In order to prove that $\pr(\tau > \sigma_y) > 0$ holds, we will show that the event $\{\tau > \sigma_y\}$ is a superset of two increasing sets, both of which have nonzero probability. To do that, we fix some $T > 0$ such that $\pr(\tau \ge T) > 0$, which is possible by the assumption of the lemma, and introduce the two sets 
    $$
    A \coloneqq \{\tau < T\} \text{ \ and \ } B \coloneqq \{\sigma_y \ge T\}
    $$ 
    in the notation of \eqref{hitting_time_less}, \eqref{hitting_time_boundary}.
    We claim that these sets are both increasing. Indeed, we have 
    \begin{align*}
        \omega_1 \leq \omega_2 &\implies \left \{t \in [0, T] : W_t(\omega_1) + x \ge b(t) \right\} \subseteq \left \{t \in [0, T] : W_t(\omega_2) + x \ge b(t) \right\} \\
        & \implies \tau(\omega_2) \leq \tau(\omega_1)
    \end{align*} 
    if $\tau(\omega_1) < T$, and therefore,
    $$
    \omega_1 \leq \omega_2, \, \omega_1 \in A \implies \omega_2 \in A,
    $$
    so that $A$ is increasing by definition.  Identical logic yields that $B$ is also increasing.  Thus, we have 
    \begin{align*}
    \pr(\tau > \sigma_y) & \ge \pr(\tau \ge T, \sigma_y < T) \\
    &= 1 - \pr(\tau < T) - \pr(\sigma_y \ge T) + \pr(\tau < T, \sigma_y \ge T) \\
    &\ge 1 - \pr(\tau < T) - \pr(\sigma_y \ge T) + \pr(\tau < T) \, \pr(\sigma_y \ge T) \\
    &= \pr(\tau \ge T) \, \pr(\sigma_y < T) \\
    &> 0,
    \end{align*} 
    where the inequality on the third line follows from the FKG inequality \eqref{FKG_statement}, and the last inequality from the fact that $\pr(\sigma_y < T) > 0$ for every $T > 0$. This concludes the proof.
\end{proof}

\section{Future research}\label{sec_future_research}

We conclude by formulating three possible extensions of our results, which we find interesting to explore in future research.

The first extension would be to investigate more general conditions on the reward function $f(\cdot)$, under which the statement of Corollary \ref{completely_solved_game} still holds. Note that for Nash equilibria in which both players receive finite expected payoffs, our proof relies heavily on Lemma \ref{boundary_stop_times}, which, under the linearity assumption, shows that the first hitting times of some (upper) boundaries are optimal stopping times. A natural question is under what assumptions, in addition to \ref{assm_1}--\ref{assm_3}, this still holds, i.e., one-sided hitting times are optimal.

The second extension would be to consider either a finite-time horizon analogue of the Dynkin game from Section \ref{model}, or a game with an extra discounting factor introduced to the model. We believe that both these modifications will create non-trivial Nash equilibria with finite expected payoffs.

The third extension, which is the most interesting from a probability theory perspective, is to consider the same game with a bounded reward function. In this case, only equilibria with finite expected payoffs are possible, and we firmly believe that, under appropriate conditions on the reward functions, there must exist non-trivial equilibria of this type. However, the analysis of such a game seems to be very demanding. Even for the results of the present paper, quite elaborate probabilistic analysis is required. The bounded reward setting is more challenging; it seems natural to consider the symmetric version of the optimal stopping problem \eqref{value_func_as_discounting} with ``endogenous discounting'' and apply fixed-point arguments to obtain a Nash equilibrium. This direction, as well as the finite-horizon game, are currently under investigation.
\\

\noindent
\textbf{Acknowledgments}

We are most grateful to Ioannis Karatzas for his very careful readings and many valuable suggestions. We thank Erik Ekström for suggesting to us a problem, still open, that motivated us to work along the present lines. Many thanks go to Milind Hegde and Shalin Parekh for pointing out relevant results that we used in the paper. We also thank the anonymous referees for very valuable comments that helped significantly improve the exposition of the paper, and for motivating us to explore the mixed strategy setting.

The authors gratefully acknowledge support from the National Science Foundation under grant NSF-DMS-20-04997.

% \bibliography{references}

\printbibliography

\end{document}